\documentclass[a4paper]{amsart}
\numberwithin{equation}{section}
\newtheorem{theorem}{Theorem}[section]
\newtheorem{corollary}{Corollary}[section]

\newtheorem{example}{Example}[section]

\theoremstyle{remark}
\newtheorem{remark}{Remark}[section]
\usepackage{hyperref}
\usepackage{color}
\usepackage{graphics}
\usepackage{graphicx}
\usepackage{epsf}
\usepackage{amsfonts}
\title[On starlike functions]
 {On starlike functions related with the convex conic domain}
\begin{document}
%----------------------------------------------------------------------------
\begin{abstract}
In the present paper, we study a new subclass $\mathcal{M}_p(\alpha,\beta)$ of $p$--valent functions and obtain some inequalities concerning the coefficients for the desired class.
Also, by use of the Hadamard product, we define a general operator and find a condition such that it belongs to the class $\mathcal{M}_p(\alpha,\beta)$.
\end{abstract}
%----------------------------------------------------------------------------
\author[R. Kargar, J. Sok\'{o}{\l}]
       {R. Kargar and J. Sok\'{o}{\l}}
%----------------------------------------------------------------------------
\address{Young Researchers and Elite Club,
Ardabil Branch, Islamic Azad University, Ardabil, Iran}
       \email {rkargar1983@gmail.com {\rm (R. Kargar)}}
%----------------------------------------------------------------------------
\address{ University of Rzesz\'{o}w, Faculty of Mathematics and Natural
         Sciences, ul. Prof. Pigonia 1, 35-310 Rzesz\'{o}w, Poland}
      \email{jsokol@ur.edu.pl {\rm (J. Sok\'{o}{\l})}}
%----------------------------------------------------------------------------
%------------------------------------------------------------

\maketitle
\section{Introduction}
%=================================================================
Let $\mathcal{A}_p$ denote the class of functions $f$ of the
form:
\begin{equation*}
  f(z)=z^p+ \sum_{k=p+1}^{\infty}a_{k}z^{k}\quad p\in \mathbb{N}:=\{1,2,3,\ldots\},
\end{equation*}
which are analytic and $p$--valent in the unit disk
$\Delta=\{z\in \mathbb{C} : |z|<1\}$. Further, we write that
$\mathcal{A}_1=\mathcal{A} $. %As usual, we denote by $\mathcal{S}$
%the subclass of $\mathcal{A}$, consisting of functions which are
%univalent in $\Delta$.
A function $f(z)\in\mathcal{A}_p$ is said to
be in the class $\mathcal{M}_p(\alpha,\beta)$, if it satisfies
\begin{equation}\label{eq1}
  {\rm Re}\left(\frac{zf'(z)}{f(z)}\right)< \beta \left|\frac{zf'(z)}{f(z)}-p\right|+p\alpha\quad (z\in\Delta),
\end{equation}
for some $\beta\leq 0$ and $\alpha>1$.
Note that \eqref{eq1} follows that
\begin{equation}\label{c}
    {\rm Re}\left(\frac{zf'(z)}{f(z)}\right)< p\alpha\quad (z\in\Delta),
\end{equation}
because $\beta$ is negative or $0$. Moreover, if we write
\eqref{eq1} in the following equivalent form
\begin{equation*}
    \frac{|F(z)-p|}{{\rm Re}\left\{p\alpha-F(z)\right\}}<\frac{1}{-\beta}\quad
    (z\in\Delta),
\end{equation*}
where
\begin{equation*}
  F(z):=\frac{zf'(z)}{f(z)}\quad
    (z\in\Delta),
\end{equation*}
then we see that the relation between the distance $F(z)$ from the
focus $p$ and the distance $F(z)$ from the directrix ${\rm
Re}\left\{w\right\}=p\alpha$ depends on the eccentricity $-1/\beta$.
Therefore, if $f(z)\in\mathcal{M}_p(\alpha,\beta)$, then $F(z)$,
$z\in\Delta$ lies in a convex conic domain: elliptic when
$\beta<-1$, parabolic when $\beta=-1$ and hyperbolic when
$-1<\beta<0$, or $F(\Delta)$ is the half--plane \eqref{c}, for
$\beta=0$.

The class $\mathcal{M}_p(\alpha,\beta)$ cover many subclasses
considered earlier by various authors \cite{AHRS, NO, OJ, UGS}. We remark that the class $\mathcal{M}_1(\alpha,\beta)=\mathcal{MD}(\alpha,\beta)$
was investigated earlier by J. Nishiwaki and S. Owa \cite{NO2007}.

In this work we shall be mainly concerned with functions
$f\in\mathcal{A}_p$ of the form
\begin{equation}\label{mu}
\left(\frac{z^p}{f(z)}\right)^\mu=1-\sum_{k=p}^{\infty}b_kz^k \quad
(\mu>0, z\in\Delta\cup\{1\}),
\end{equation}
where $(z^p/f(z))^\mu$ represents principal powers (i.e. the
principal branch of $(z^p/f(z))^\mu$ is chosen).

This paper is organized as follows. In Section \ref{sec2}, we present some inequalities for the class  $\mathcal{M}_p(\alpha,\beta)$. In Section \ref{sec3}, by use of the Hadamard product and applying the generalized Bessel function and the Gaussian hypergeometric function we introduce a new operator which we denote by $\mathcal{I}_{c,d}^{a,b}(p,e,\delta)(z)$. As an application, we prove that the operator $\mathcal{I}_{c,d}^{a,b}(p,e,\delta)(z)$ belongs to the class $\mathcal{M}_p(\alpha,\beta)$.

%-----------------------------------------------------------------------------------------------------------------------------------
\section{Main Results}\label{sec2}
%-----------------------------------------------------------------------------------------------------------------------------------
%Theoremm 2.1*******************************************************************
Our first result is contained in the following:
\begin{theorem}\label{t21}
Let the function $f$  be in the class $\mathcal{M}_p(\alpha,\beta)$
and let $f$ be of the form \eqref{mu} for some $b_k$ such that
\begin{equation*}
    b_k\geq0\quad for \quad k\in\{p,p+1,p+2,\ldots\} \quad and \quad
    \sum_{k=p}^{\infty}b_k<1.
\end{equation*}
Then
\begin{equation}\label{1t21}
    \sum_{k=p}^{\infty}[p\mu(\alpha-1)+k(1-\beta)]b_k\leq
    p\mu(\alpha-1),
\end{equation}
where $\beta\leq 0$, $\mu>0$ and $\alpha>1$.
\end{theorem}
%***********************************
\begin{proof}
Let $f\in\mathcal{M}_p(\alpha,\beta)$. A simple calculation gives us
\begin{equation}\label{e2}
z\frac{{\rm d}}{{\rm
d}z}\left(\frac{z^p}{f(z)}\right)^\mu=\mu\left[p\left(\frac{z^p}{f(z)}\right)^\mu
-\left(\frac{z^p}{f(z)}\right)^{\mu+1}\frac{f'(z)}{z^{p-1}}\right].
\end{equation}
Thus, by use of the above relation \eqref{e2}, we have
\begin{equation*}
  {\rm Re}\left(\frac{zf'(z)}{f(z)}\right)< \beta
\left|\frac{zf'(z)}{f(z)}-p\right|+p\alpha
\end{equation*}
 if and only if
\begin{equation*}
  {\rm Re}\left(p+\frac{-\frac{z}{\mu} \frac{{\rm d}}{{\rm d}z}\left(z^p/f(z)\right)^\mu}{\left(z^p/f(z)\right )^\mu}\right)
< \beta\left|\frac{-\frac{z}{\mu} \frac{{\rm d}}{{\rm
d}z}\left(z^p/f(z)\right)^\mu}{\left(z^p/f(z)\right
)^\mu}\right|+p\alpha.
\end{equation*}
 Since $f$ is of the form \eqref{mu}, the last
inequality may be equivalently written as
\begin{equation}\label{e33}
{\rm
Re}\left(p+\frac{1}{\mu}\frac{\sum_{k=p}^{\infty}kb_kz^k}{1-\sum_{k=p}^{\infty}b_kz^k}\right)
<
\frac{\beta}{\mu}\left|\frac{\sum_{k=p}^{\infty}kb_kz^k}{1-\sum_{k=p}^{\infty}b_kz^k}\right|+p\alpha.
\end{equation}
Now suppose that $z\in\Delta$ is real and tends to $1^-$ through
reals, then from the last inequality \eqref{e33}, we get
\begin{equation*}
\mu p+\frac{\sum_{k=p}^{\infty}kb_k}{1-\sum_{k=p}^{\infty}b_k}\leq
\beta
\left|\frac{\sum_{k=p}^{\infty}kb_k}{1-\sum_{k=p}^{\infty}b_k}\right|+p\mu\alpha
\end{equation*}
or equivalently
\begin{equation*}
\mu p+\frac{\sum_{k=p}^{\infty}kb_k}{1-\sum_{k=p}^{\infty}b_k}\leq
\beta\frac{\sum_{k=p}^{\infty}kb_k}{1-\sum_{k=p}^{\infty}b_k}+p\mu\alpha,
\end{equation*}
which gives \eqref{1t21}.
\end{proof}

%remarkk 2.1******************************************************************
\begin{remark}
Taking $p=\mu=1$ in the above Theorem \ref{t21}, we get the result that has been obtained recently by
Aghalary et al. \cite[Theorem 2.1 with $n=1$]{AEO}.
\end{remark}

Next we derive the following:
%Theoremm 2.2*******************************************************************
\begin{theorem}\label{t22}
Let $f\in\mathcal{A}_p$ be of the form \eqref{mu} with $\mu>0$.
If
\begin{equation}\label{3t22}
\sum_{k=p}^{\infty}[p\mu(\alpha-1)+k(1-\beta)]| b_k|< p\mu(\alpha-1),
\end{equation}
then $f$ is in the class $\mathcal{M}_p(\alpha,\beta)$, where $\beta\leq 0$ and $\alpha>1$.
\end{theorem}
%----------------------------------------------------------------------
\begin{proof}
In the proof of Theorem \ref{t21}, we saw the following inequality
\begin{equation*}
  {\rm Re}\left(\frac{zf'(z)}{f(z)}\right)-p\alpha< \beta
\left|\frac{zf'(z)}{f(z)}-p\right|,
\end{equation*}
 is equivalent to
\begin{equation*}\label{5t22}
  {\rm Re}\left(\frac{\sum_{k=p}^{\infty}kb_kz^k}{1-\sum_{k=p}^{\infty}b_kz^k}\right)-p\mu(\alpha-1)< \beta\left|
 \frac{\sum_{k=p}^{\infty}kb_kz^k}{1-\sum_{k=p}^{\infty}b_kz^k}\right|.
\end{equation*}
Thus, to show that $f$ is in the class $\mathcal{M}_p(\alpha,\beta)$, it
suffices to prove that
\begin{equation}\label{6t22}
  {\rm Re}\left(\frac{\sum_{k=p}^{\infty}kb_kz^k}{1-\sum_{k=p}^{\infty}b_kz^k}\right)
  -\beta\left|
 \frac{\sum_{k=p}^{\infty}kb_kz^k}{1-\sum_{k=p}^{\infty}b_kz^k}\right|
 <p\mu(\alpha-1).
\end{equation}
Note that from  \eqref{3t22}, we obtain that
\begin{equation*}
 1-\sum_{k=p}^{\infty}|b_k|>0.
\end{equation*}
Therefore one can rewrite \eqref{3t22} in the following equivalent form
\begin{equation}\label{7t22}
    \frac{\sum_{k=p}^{\infty}k|b_k|}{1-\sum_{k=p}^{\infty}|b_k|}
    -\beta\frac{\sum_{k=p}^{\infty}k|b_k|}{1-\sum_{k=p}^{\infty}|b_k|}
    <p\mu(\alpha-1).
\end{equation}
Because $\beta\leq 0$, we have
\begin{equation}\label{8t22}
    {\rm Re}\left(\frac{\sum_{k=p}^{\infty}kb_kz^k}{1-\sum_{k=p}^{\infty}b_kz^k}\right)
    -\beta\left|
    \frac{\sum_{k=p}^{\infty}kb_kz^k}{1-\sum_{k=p}^{\infty}b_kz^k}\right|
    \leq\frac{\sum_{k=p}^{\infty}k|b_k|}{1-\sum_{k=p}^{\infty}|b_k|}
    -\beta\frac{\sum_{k=p}^{\infty}k|b_k|}{1-\sum_{k=p}^{\infty}|b_k|}.
\end{equation}
 Then \eqref{7t22} and \eqref{8t22} immediately follow
\eqref{6t22} and therefore, $f\in\mathcal{M}_p(\alpha,\beta)$.
\end{proof}
%Corollaryy 2.1********************************************************
\begin{corollary}
Assume that $f\in \mathcal{A}$ and $(z/f(z))^\mu
=1-\sum_{k=1}^{\infty} b_k z^k$ with $\mu>0$. If the function $f$
satisfies the condition
\begin{equation*}
\sum_{k=1}^{\infty}[\mu(\alpha-1)+k(1-\beta)]| b_k|<
\mu(\alpha-1),
\end{equation*}
then $f$ is in the class $\mathcal{MD}(\alpha,\beta)$, where $\beta\leq 0$ and $\alpha>1$.
\end{corollary}
%remarkk 2.1**************************************************************
\begin{remark}
The  case $p=\mu=1$ in Theorem \ref{t22} has been obtained recently
by Aghalary et al. \cite[Theorem 2.2]{AEO}.
\end{remark}
%theoremm 2.3***********************************************************************
\begin{theorem}\label{t23}
A function $f$ of the form $f(z)=z^p+
\sum_{k=p+1}^{\infty}a_{k}z^{k}$ is in the class
$\mathcal{M}_p(\alpha,\beta)$, if
\begin{equation}\label{1t23}
\sum_{k=p+1}^{\infty}[p\alpha+\beta+k(1-\beta)]| a_k| < p(\alpha-1).
\end{equation}
\end{theorem}
\begin{proof}
Assume that $f\in \mathcal{A}_p$. Then by use of the definition of $\mathcal{M}_p(\alpha,\beta)$, we have
\begin{equation*}
  {\rm Re}\left(\frac{zf'(z)}{f(z)}\right)-p\alpha< \beta
\left|\frac{zf'(z)}{f(z)}-p\right|
\end{equation*}
 if and only if
\begin{equation*}
  {\rm Re}\left(\frac{p+\sum_{k=p+1}^{\infty}ka_kz^{k-p}}
  {1+\sum_{k=p+1}^{\infty}a_kz^{k-p}}\right)-p\alpha< \beta
  \left|\frac{\sum_{k=p+1}^{\infty}(k-1)a_kz^{k-p}}
  {1+\sum_{k=p+1}^{\infty}a_kz^{k-p}}\right|.
\end{equation*}
Thus, it suffices to show that
\begin{equation}\label{2t23}
  {\rm Re}\left(\frac{p+\sum_{k=p+1}^{\infty}ka_kz^{k-p}}
  {1+\sum_{k=p+1}^{\infty}a_kz^{k-p}}\right)- \beta
  \left|\frac{\sum_{k=p+1}^{\infty}(k-1)a_kz^{k-p}}
  {1+\sum_{k=p+1}^{\infty}a_kz^{k-p}}\right|<p\alpha.
\end{equation}
Note that from \eqref{1t23}, we have
\begin{equation*}
  1-\sum_{k=p+1}^{\infty}|a_k|>0,
\end{equation*}
thus from \eqref{1t23} we obtain that
\begin{equation}\label{3t23}
  \frac{p+\sum_{k=p+1}^{\infty}k|a_k|}{1-\sum_{k=p+1}^{\infty}|a_k|}
  -\beta\frac{\sum_{k=p+1}^{\infty}(k-1)|a_k|}{1-\sum_{k=p+1}^{\infty}|a_k|}<p\alpha.
\end{equation}
Moreover, we have
\begin{equation}\label{4t23}
    {\rm Re}\left(\frac{p+\sum_{k=p}^{\infty}ka_kz^{k-p}}
    {1+\sum_{k=p}^{\infty}a_kz^{k-p}}\right)- \beta
  \left|\frac{\sum_{k=p}^{\infty}(k-1)a_kz^{k-p}}{1+\sum_{k=p}^{\infty}a_kz^{k-p}}\right|
\end{equation}
\begin{equation*}
  \leq\frac{p+\sum_{k=p+1}^{\infty}k|a_k|}
  {1-\sum_{k=p+1}^{\infty}|a_k|}-\beta\frac{\sum_{k=p+1}^{\infty}(k-1)|a_k|}
  {1-\sum_{k=p+1}^{\infty}|a_k|}.
\end{equation*}
Therefore, \eqref{3t23} and \eqref{4t23} follow \eqref{2t23}. This
completes the proof.
\end{proof}
Putting $p=1$ in Theorem \ref{t23} we have:
%corollary 2.2*******************************************
\begin{corollary}\label{c22}
If $f\in \mathcal{A}$ satisfies
\begin{equation*}
\sum_{k=2}^{\infty}[\alpha+\beta+k(1-\beta)]| a_k| < \alpha-1,
\end{equation*}
for some $\alpha>1$ and $\beta\leq 0$, then $f\in\mathcal{MD}(\alpha,\beta)$.
\end{corollary}
At the end of this section, by Theorem \ref{t23} we consider an example for the class $\mathcal{M}_p(\alpha,\beta)$.
%Example*******************************************************************
\begin{example}
  Define the function $f\in \mathcal{A}_p$ as follows
  \begin{equation*}
    f(z)=z^p+\sum_{k=p+1}^{\infty}\frac{p(p-1)(\alpha-1)e^{i\theta}}{[p\alpha+\beta+k(1-\beta)]k(k-1)}z^k,
  \end{equation*}
  where $\alpha>1$, $\beta\leq 0$ and $\theta\in\mathbb{R}$. Then the coefficients inequality \eqref{1t23}
yields
   \begin{eqnarray*}
    \sum_{k=p+1}^{\infty}[p\alpha+\beta+k(1-\beta)]| a_k| &=& p(p-1)(\alpha-1)\sum_{k=p+1}^{\infty}\frac{1}{k(k-1)}\\
     &=& (p-1)(\alpha-1)\\
     &<& p(\alpha-1),
   \end{eqnarray*}
  which shows $f\in\mathcal{M}_p(\alpha,\beta)$.
\end{example}

%===================================================
\section{Applications}\label{sec3}
%==================================================

In \cite{PD} Porwal and Dixit considered the function
$\mathcal{U}_{d,e,\delta}$ defined by the transformation
\begin{equation*}
\mathcal{U}_{d,e,\delta}(z)=2^d\Gamma\left(d+\frac{e+1}{2}\right)z^{-d/2}
w_{d,e,\delta}(z^{1/2}),
\end{equation*}
where $w_{d,e,\delta}$ is called the generalized
Bessel function of the first kind of order $d$ and has the familiar representation
\begin{equation*}
w(z)=w_{d,e,\delta}(z)=\sum_{k=0}^{\infty}\frac{(-1)^k\delta^k}
{k!\Gamma(d+k+\frac{e+1}{2})}\left(\frac{z}{2}\right)^{2k+d}\quad
z,d,e,\delta\in\mathbb{C}.
\end{equation*}
It is easy to see the function $\mathcal{U}_{d,e,\delta}$ has the following
representation
\begin{equation}\label{u1}
\mathcal{U}_{d,e,\delta}(z)=\sum_{k=0}^{\infty}\frac{(-1)^k(\delta/4)^k}
{\left(d+\frac{e+1}{2}\right)_k}\frac{z^k}{k!},
\end{equation}
where $d+\frac{e+1}{2}\neq 0,-1,-2,\ldots$ and  and $(x)_n$ is the
Pochhammer symbol defined by
\[ (x)_n:=\left\{%
\begin{array}{ll}
    1, & \hbox{$(n=0)$;} \\
    x(x+1)(x+2)\dots(x+n-1), & \hbox{$(n\in\mathbb{N})$.}
\end{array}%
\right.    \] We remark that the function
$\mathcal{U}_{d,e,\delta}(z)$ is analytic on $\mathbb{C}$.

The Gaussian hypergeometric function $F (a, b; c; z)$ given by
\begin{equation}\label{F}
F(a, b; c;
z)=\sum_{k=0}^{\infty}\frac{(a)_k(b)_k}{(c)_k(1)_k}z^k\quad
(z\in\Delta),
\end{equation}
where $a,b,c\in\mathbb{C}$ and $c\neq 0,-1,-2,\ldots$.
We note that $F(a, b; c; 1)$ converges for ${\rm Re}(a-b-c)>0$ and is related to the Gamma
function by
\begin{equation*}
  F(a, b; c; 1)=\frac{\Gamma(c-a-b)\Gamma(c)}{\Gamma(c-a)\Gamma(c-b)}\quad
{\rm Re}(c-a-b)>0.
\end{equation*}
Now by Using
\eqref{u1} and \ref{F} we introduce a new function
$\mathcal{I}_{c,d}^{a,b}(p,e,\delta)(z):\mathcal{A}_p\rightarrow
\mathcal{A}_p$ defined by
\begin{equation*}
\mathcal{I}_{c,d}^{a,b}(p,e,\delta)(z)=z^p(F (a,
b;c;z)*\mathcal{U}_{d,e,\delta}(z)),
\end{equation*}
where $"*"$ is the well--known Hadamard product. The function
$\mathcal{I}_{c,d}^{a,b}(p,e,\delta)(z)$ has the following
representation
\begin{equation*}
\mathcal{I}_{c,d}^{a,b}(p,e,\delta)(z)=z^p+\sum_{k=p+1}^{\infty}(-1)^{k-p}
\frac{(a)_{k-p}(b)_{k-p}(\delta/4)^{k-p}}{(c)_{k-p}\left(d+\frac{e+1}{2}\right)_{k-p}[(1)_{k-p}]^2}z^{k}.
\end{equation*}
We set $\mathcal{I}_{c,d}^{a,b}(1,e,\delta)(z)\equiv \mathcal{I}_{c,d}^{a,b}(e,\delta)(z)$. Applying Theorem \ref{t23} we have the following
Theorem:

%thoremm 3.1*********************************************************************
\begin{theorem}\label{t31}
Let $a,b\in \mathbb{C}\backslash \{0\}$ and $e,d\in \mathbb{C}$. Also, assume that
$\delta, d+\frac{e+1}{2}>0$ and $c$ be a real number such that
$c>|a|+|b|+1$. Then
$\mathcal{I}_{c,d}^{a,b}(p,e,\delta)(z)\in\mathcal{M}_p(\alpha,\beta)$ if
\begin{equation}\label{e12}
\sum_{k=p+1}^{\infty}[p\alpha+\beta+k(1-\beta)]
\frac{|(a)_{k-p}(b)_{k-p}|(\delta/4)^{k-p}}{(c)_{k-p}\left(d+\frac{e+1}{2}\right)_{k-p}[(k-p)!]^2}
< p(\alpha-1).
\end{equation}
\end{theorem}

\begin{proof}
Let $f(z)=z^p+\sum_{k=p+1}^{\infty}a_kz^k\in \mathcal{A}_p$. By virtue of
Theorem \ref{t23}, it suffices to show that
\begin{equation*}
\sum_{k=p+1}^{\infty}[p\alpha+\beta+k(1-\beta)]
\left|\frac{(-1)^{k-p}(a)_{k-p}(b)_{k-p}(\delta/4)^{k-p}}{(c)_{k-p}\left(d+\frac{e+1}{2}\right)_{k-p}[(1)_{k-p}]^2}\right|
< p(\alpha-1).
\end{equation*}
Some reductions give \eqref{e12}
\end{proof}
Setting $p=1$ in Theorem \ref{t31}, we have:
\begin{corollary}\label{c31}
If $a,b\in \mathbb{C}\backslash \{0\}$, $e,d\in \mathbb{C}$, $\delta, d+\frac{e+1}{2}>0$ and $c$ be a real number such that
$c>|a|+|b|+1$, then a sufficient condition for
$\mathcal{I}_{c,d}^{a,b}(e,\delta)(z)\in\mathcal{MD}(\alpha,\beta)$ is that
\begin{equation*}
\sum_{k=2}^{\infty}[\alpha+\beta+k(1-\beta)]
\frac{|(a)_{k-1}(b)_{k-1}|(\delta/4)^{k-1}}{(c)_{k-1}\left(d+\frac{e+1}{2}\right)_{k-1}[(k-1)!]^2}
< \alpha-1.
\end{equation*}
\end{corollary}
If we take $\beta=0$ in Corollary \ref{c31}, we have the following result:
\begin{corollary}\label{c32}
If $a,b\in \mathbb{C}\backslash \{0\}$, $e,d\in \mathbb{C}$, $\delta, d+\frac{e+1}{2}>0$ and $c$ be a real number such that
$c>|a|+|b|+1$, then a sufficient condition for
$\mathcal{I}_{c,d}^{a,b}(e,\delta)(z)\in\mathcal{M}(\alpha)$ is that
\begin{equation*}
\sum_{k=2}^{\infty}[\alpha+k]
\frac{|(a)_{k-1}(b)_{k-1}|(\delta/4)^{k-1}}{(c)_{k-1}\left(d+\frac{e+1}{2}\right)_{k-1}[(k-1)!]^2}
< \alpha-1.
\end{equation*}
The class $\mathcal{M}(\alpha)$ was considered by Uralegaddi et al.
\cite{UGS}, Nishiwaki and Owa \cite{NO}, and Owa and Nishiwaki
\cite{OJ}.
\end{corollary}

\end{document}